\documentclass[11pt]{amsart}
\usepackage[utf8]{inputenc}
\usepackage[utf8]{inputenc}
\usepackage{amsfonts,amssymb}
\usepackage{graphicx,tikz}
\usepackage{tikz-cd} 
\usepackage{float}
\usepackage{amsmath, amsthm, amsfonts}
\usepackage{geometry}
\usepackage{hyperref}
\usepackage{enumitem}  
\usepackage[capitalise]{cleveref}
\usepackage{comment}
\usepackage{enumitem}
\usepackage{amsrefs}
\usepackage{nicefrac}
\usepackage[official]{eurosym}
\usepackage{graphicx}
\usepackage{nomencl}
\usepackage{amsfonts}
\usepackage{hyperref}
\usepackage{amssymb}
\usepackage{fancyhdr}
\usepackage{amscd}
\usepackage[english]{babel}

\newtheorem{theorem}{Theorem}[section]
\newtheorem{thm}[theorem]{Theorem}
\newtheorem*{thmA*}{Theorem A}
\newtheorem*{thmB*}{Theorem B}
\newtheorem*{thmC*}{Theorem C}
\newtheorem*{thmD*}{Theorem D}
\newtheorem*{thmE*}{Theorem E}
\newtheorem{lem}[theorem]{Lemma}
\newtheorem{pro}[theorem]{Proposition}
\newtheorem{cor}[theorem]{Corollary}

\newcommand{\N}{\mathbb{N}}

\newcommand{\hyp}{Z_{\infty}}

\title[Nilpotent Graph]{Neighborhoods, connectivity, and diameter of the nilpotent graph of a finite group}
\author[Delizia, Gaeta, Lewis, and Monetta]{Costantino Delizia, Michele Gaeta, Mark L. Lewis, Carmine Monetta }

\address{Dipartimento di Matematica, Universit\`a di Salerno, Fisciano 84084 (SA), Italy}
\email{(Delizia) cdelizia@unisa.it, (Gaeta) migaeta@unisa.it, (Monetta) cmonetta@unisa.it}

\address{Department of Mathematical Sciences,
Kent State University, Kent, OH  44242, U.S.A.}
\email{(Lewis) lewis@math.kent.edu}

\begin{document}

\begin{abstract}
The nilpotent graph of a group $G$ is the simple and undirected graph whose vertices are the elements of $G$ and two distinct vertices are adjacent if they generate a nilpotent subgroup of $G$. Here we discuss some topological properties of the nilpotent graph of a finite group $G$. Indeed, we characterize finite solvable groups whose closed neighborhoods are nilpotent subgroups. Moreover, we study the connectivity of the graph $\Gamma(G)$ obtained removing all universal vertices from the nilpotent graph of $G$. Some upper bounds to the diameter of $\Gamma(G)$ are provided when $G$ belongs to some classes of groups.
\vspace{0,3cm} 


\end{abstract}

\maketitle

\section{Introduction}
\noindent In this paper, all groups are finite.  The study of graphs associated with groups has become an increasingly important area of research in modern algebra.  In this paper, we focus on groups and the graphs that can be constructed from them. For a more comprehensive overview of this area and related open problems, we refer the reader to \cites{BLN, Cameron, GLM, GM, Shahverdi}.

One particularly significant graph is the {\it nilpotent graph} of a group $G$, which is the simple undirected graph where the vertices correspond to the elements of $G$, and two distinct vertices $x$ and $y$ are adjacent if and only if the subgroup $\langle x, y\rangle$ generated by $x$ and $y$ is nilpotent.  The concept of the nilpotent graph first emerged implicitly in the series of papers \cites{DMN, DMN2, DMN3}, where the authors examine groups whose nilpotent graphs consist of complete connected components.  Later, in \cite{AZ}, the authors studied the properties of the complement of the nilpotent graph.  More recently, some specific properties of this graph have been explored in \cite{torres}.

The nilpotent graph is part of a broader family of graphs that have been studied in the context of finite groups. For instance, recent work has focused on the properties of neighborhoods within these graphs (see \cites{ADM, ALMM, LM}).  For each element $x \in G$, the nilpotent neighborhood of $x$, denoted by ${\rm Nil}_G(x)$, is the set of elements $y \in G$ such that the subgroup $\langle x, y\rangle$ is nilpotent.  Understanding the structure of these neighborhoods is a crucial step in classifying groups based on their nilpotent graph.  As shown in \cite{AZ}, the subset ${\rm Nil}_G(x)$ is not a subgroup in general.  A still open question is the classification of groups for which ${\rm Nil}_G(x)$ is a subgroup of $G$ for every $x \in G$.  Such groups, which following \cite{AZ} we call {\it $\mathfrak n$-groups}, exhibit special properties in relation to their nilpotent subgroups.  While the classification of simple $\mathfrak n$-groups has been achieved in earlier works, such as \cite{AZ}, the study of solvable $\mathfrak n$-groups remains a relatively unexplored area.

In this paper, we investigate solvable $\mathfrak n$-groups, establishing a full classification of those $\mathfrak n$-groups for which every nilpotent neighborhood is a nilpotent subgroup: indeed we prove that the class of Frobenius groups having a nilpotent Frobenius complement is the only class of solvable groups for which this condition holds.

\begin{thmA*}\label{introthm:nilp}
Let $G$ be a solvable group with trivial center. Then ${\rm Nil}_G(x)$ is a nilpotent subgroup of $G$ for every nontrivial element $x \in G$ if and only if $G$ is a Frobenius group with nilpotent Frobenius complement.
\end{thmA*}

It is worth mentioning that the hypothesis of Theorem A is quite natural in the context of graph theoretic problems.  Indeed, having a  nilpotent group for ${\rm Nil}_G (x)$ coincides with the requirement that every closed neighborhood of the graph is a clique. 

Another interesting graph property is connectivity.  The study of connectivity in group-related graphs has a long history (see \cites{CLSTU, CLSTU1, MP, parker}).  The first step is to detect the set of universal vertices of the graph; namely, those vertices that are adjacent to all of the other vertices of the graph.

When $G$ is any group, the set of universal vertices in the nilpotent graph for $G$ coincides with the hypercenter $\hyp(G)$ of $G$ (see Proposition 2.1 in \cite{AZ}).  Thus, we consider the graph $\Gamma (G)$, which is defined as the induced subgraph of the nilpotent graph on the set $G \setminus \hyp(G)$.  First, we show that this graph can be computed via $G/\hyp(G)$.

\begin{thmB*}
For any group $G$ the number of connected components of $\Gamma(G)$ equals the number of connected components of $\Gamma(G/\hyp(G))$, and there is a correspondence between the connected components of $\Gamma(G)$ and $\Gamma(G/\hyp(G))$ that maps connected components of diameter $1$ to connected components of diameter $0$ or $1$ and preserves the diameter of connected components whose diameter is greater than $1$. 
\end{thmB*}

For solvable groups, we show that $\Gamma(G)$ is disconnected exactly when $G/\hyp(G)$ is Frobenius or $2$-Frobenius. 

\begin{thmC*}\label{introthm:disconnected}
Let $G$ be a non-nilpotent solvable group. Then $\Gamma (G)$ is disconnected if and only if $G/\hyp(G)$ is a Frobenius group or a $2$-Frobenius group.
\end{thmC*}

We see that the structure of the quotient group $G/\hyp(G)$ plays a pivotal role in determining the connectivity of $\Gamma(G)$.   In fact, we obtain this result starting from the relationship between $\Gamma(G)$ and the commuting graph $\Gamma_{\rm comm}(G)$, which is another widely studied graph associated with finite groups.  In $\Gamma_{\rm comm} (G)$ the vertices are the non-central elements of $G$, and two distinct vertices are adjacent if and only if the corresponding elements commute.  The graphs $\Gamma (G)$ and $\Gamma_{\rm comm} (G)$ share the same vertex set if and only if the center of $G$ coincides with the hypercenter of $G$, in which case $\Gamma_{\rm comm}(G)$ is a subgraph of $\Gamma(G)$.  We note that Theorem C shares a certain similarity with Theorem 1.2 of \cite{BCCHLLP}.  

It makes sense to ask when these two graphs are equal.  In fact, these two graphs coincide if and only if $G$ is an $A$-group, i.e., a group whose Sylow subgroups are abelian.  

\begin{thmD*}\label{introprop:agroup}
Let $G$ be a group. Then $\Gamma(G)$ coincides with $\Gamma_{\rm comm}(G)$ if and only if $G$ is an $A$-group.
\end{thmD*}

We prove that if $\Gamma(G)$ is connected, then its diameter is bounded.  In fact, building upon results of Morgan and Parker in \cite{MP} and \cite{parker}, one may show that the diameter of any connected component of $\Gamma(G)$ is at most $10$. (We note that this result appears as \cite[Proposition 7.6]{BLN}, but we include it here for completeness). This bound can be reduced to $8$ when $G$ is solvable and $\Gamma (G)$ is connected.  When $G$ is solvable and $\Gamma (G)$ is disconnected, we can reduce it even further.  

\begin{thmE*}
Let $G$ be a group.  Then the following are true:
\begin{enumerate}
\item[$(1)$] The connected components of $\Gamma (G)$ have diameter at most $10$.
\item[$(2)$] If $G$ is solvable and $\Gamma (G)$ is connected, then $\Gamma (G)$ has diameter at most $8$.
\item[$(3)$] If $G$ is solvable and $\Gamma (G)$ is disconnected, then one connected component of $\Gamma (G)$ has diameter at most $5$ and the remaining connected components have diameter at most $2$. 
\end{enumerate}  
\end{thmE*}

In this paper, we are going to focus on the solvable case, so we do not determine when $\Gamma (G)$ is disconnected when $G$ is nonsolvable or even nonabelian simple.  And we will not say anything further about the diameter of $\Gamma (G)$ in the nonsolvable case.

When $G$ is solvable, we will present an example of a solvable group where $\Gamma (G)$ is connected and has diameter $7$.  At this time, we do not have any examples where diameter $8$ is obtained, so there is a gap between the bound on the diameter of $\Gamma (G)$ and the examples we can obtain.

We will see that we can obtain tighter bounds on the diameter of $\Gamma (G)$ by imposing extra conditions on $G$ and in those cases we are able to obtain the bound.  For example, when $G$ is an $A$-group, we will show that ${\rm diam} (\Gamma (G)) \le 6$.  When $G$ is a semidirect product of a normal cyclic subgroup by an abelian one, we show that ${\rm diam} (\Gamma (G)) \le 4$.  When $G$ is solvable and ${\rm Fit} (G)$ is cyclic, we show that ${\rm diam} (\Gamma (G)) \le 5$, and when $G$ is a $\{ p, q \}$-group, we show that ${\rm diam} (\Gamma (G)) \le 6$.

\section{Nilpotent graphs whose neighborhoods are cliques}\label{N}

\noindent The goal of this section is to classify the groups for which every nilpotent neighborhood is a nilpotent subgroup.  We will start by stating some preliminary results concerning the closed neighborhood of an element in the nilpotent graph. We begin with the following lemma.

\begin{lem}\label{lem:intersection}
Fix an element $x \in G$ of order $p_1^{\alpha_1} \cdots p_k^{\alpha_k}$ and let $x_i$ be the $p_i$-component of $x$ where $p_1, \dots, p_k$ are distinct primes and $\alpha_i$ are positive integers. Then
\[
{\rm Nil}_G (x) = \bigcap_{i=1}^k {\rm Nil}_G (x_i).
\]
\end{lem}

\begin{proof}
For an element $c \in {\rm Nil}_G(x)$, the subgroup $\langle x, c\rangle$ is nilpotent and as a subgroup, so is $\langle x_i, c\rangle$.  Hence, $c \in {\rm Nil}_G (x_i)$ for every $i = 1, \ldots, k$.  We obtain $c \in \bigcap_{i=1}^k {\rm Nil}_G (x_i)$, and so, ${\rm Nil}_G (x) \subseteq \bigcap_{i=1}^k {\rm Nil}_G (x_i)$.  

Conversely, assume $c \in Nil_G (x_i)$ for every $i = 1, \ldots, k$, and write $c=c_1 \cdots c_k c'$, where $c_i$ is the $p_i$-component of $c$ and $c'$ is its $\{p_1, \ldots p_k\}'$-component. Then the subgroups $\langle x_i, c_j\rangle$ and $\langle x_i, c' \rangle$ are nilpotent for every $i,j=1, \ldots, k$, as subgroups of $\langle x_i, c \rangle$. This implies that $\langle x_i, c_i \rangle$ is a $p_i$-subgroup.  In addition, $x_i$ and $c_j$ will commute when $i \neq j$, as well as, $x_i$ and $c'$. It follows that
\[
\langle x, c\rangle = \langle x_1, \ldots, x_k,c_1, \ldots, c_k, c'\rangle = \langle x_1, c_1\rangle \times \cdots \times \langle x_k, c_k\rangle \times \langle c'\rangle
\]
which is a direct product of nilpotent groups, so it is nilpotent.
\end{proof}

An immediate consequence of the Lemma~\ref{lem:intersection} is the following result, which coincides with Lemma 3.4 of \cite{AZ}.

\begin{cor}\label{cor:reduction}
Let $G$ be a group. Then $G$ is an $\mathfrak n$-group if and only if ${\rm Nil}_G (y)$ is a subgroup of $G$ for every element $y \in G$ of prime power order.
\end{cor}

In \cite{AZ}, the authors provide some classes of $\mathfrak n$-groups. For instance, they prove that every $A$-group is such a group. As a consequence, $A$-groups may be characterized as follows:

\begin{thm}\label{thm:agroup}
A group $G$ is an $A$-group if and only if ${\rm Nil}_G (x) = C_G (x)$ for every element $x \in G$ of prime power order.
\end{thm}

\begin{proof}
When $G$ is an $A$-group, the result follows by Lemma 3.5 of \cite{AZ}.
Conversely, suppose ${\rm Nil}_G (x) = C_G (x)$ for every element $x \in G$ of prime power order.  Let $P$ be a Sylow $p$-subgroup of $G$, and consider the elements $x,y \in P$. Then $\langle x,y \rangle$ is nilpotent, which implies $y \in Nil_G (x) = C_G (x)$. Therefore, $x$ commutes with $y$ and $P$ is abelian.  
\end{proof}

We now show that the class of $\mathfrak n$-group is closed under taking direct products.

\begin{lem}\label{lem:directproduct}
Let $A$ and $B$ be groups such that ${\rm Nil}_A (a)$ and ${\rm Nil}_B (b)$ are subgroups of $A$ and $B$ respectively, for every pair of elements $a \in A$ and $b \in B$. Then ${\rm Nil}_{A \times B} (x)$ is a subgroup of $A \times B$ for every element $x \in A \times B$.
\end{lem}

\begin{proof}
Consider elements $x, y \in A \times B$.  Then there exist elements $a, a' \in A$ and $b, b'\in B$ such that $x = ab$ and $y = a'b'$.  Now, $\langle x,y \rangle$ is nilpotent if and only if $\langle ab, a'b' \rangle = \langle a,a' \rangle \times \langle b,b' \rangle$ is nilpotent.  This implies that ${\rm Nil}_{A\times B} (ab) = {\rm Nil}_A (a) \times {\rm Nil}_B (b)$. The result is now clear.
\end{proof}

We next show that any Frobenius group whose Frobenius complement is an $\mathfrak n$-group is again an $\mathfrak n$-group.

\begin{lem}\label{lem:frobenius}
Let $G$ be a Frobenius group with Frobenius complement $H$. If $H$ is an  $\mathfrak n$-group, then $G$ is an  $\mathfrak n$-group.
\end{lem}

\begin{proof}
Suppose the element $x \in G$ belongs to the Frobenius kernel $K$ of $G$.  Then ${\rm Nil}_G (x) = K$ as $K$ is nilpotent.  Consider an element $x \in G \setminus K$.  Then $x$ belongs to some conjugate of $H$, say $H_1$.  Since $H_1$ and $K$ have coprime orders, it is easy to see that ${\rm Nil}_G (x) \leq H_1$, implying that ${\rm Nil}_G (x) = {\rm Nil}_{H_1} (x)$. Since $H_1$ is an $\mathfrak n$-group, the result follows.
\end{proof}

We denote by $O_{\pi} (G)$ the $\pi$-radical of $G$, that is, the maximum normal $\pi$-subgroup of $G$. Next, $O^{\pi}(G)$ denotes the $\pi$-residual of $G$, namely, the smallest normal subgroup of $G$ for which the quotient is a $\pi$-group. The commutator of two elements $x$ and $y$ of a group $G$ is the element $[x,y]=x^{-1}y^{-1}xy=x^{-1}x^y$.  We are now ready to prove the following.

\begin{thmA*}
Let $G$ be a finite solvable group with trivial center. Then ${\rm Nil}_G(x)$ is a nilpotent subgroup of $G$ for every nontrivial element $x \in G$ if and only if $G$ is a Frobenius group with nilpotent Frobenius complement.
\end{thmA*}

\begin{proof}
Let $\pi$ be the set of all primes $p \in \pi(G)$ such that $O_p (G)$ is not trivial.  Since $G$ is solvable, the set $\pi$ is not empty.  For a prime $p \in \pi$, consider an element $x_p \in O_p (G) \setminus \{1\}$.
    
Then $O^{p'} (G)$ is contained in ${\rm Nil}_G (x_p)$ and therefore it is nilpotent.  As a consequence $O^{p'} (G)$ is a Sylow $p$-subgroup, implying that for every prime $p \in \pi$ there exists a unique Sylow $p$-subgroup of $G$.  Set $K$ to be the subgroup of $G$ generated by all Sylow $p$-subgroups of $G$ for $p \in \pi$.  Note that $K$ coincides with the Fitting subgroup of $G$ because any minimal normal subgroup of $G$ is contained in the Fitting subgroup of $G$.  By the Schur-Zassenhaus Theorem, there exists a subgroup $H$ of $G$ such that $G = KH$ and $H \cap K = \{1\}$. 

To prove that $G$ is a Frobenius group with Frobenius kernel $K$ and Frobenius complement $H$, it is sufficient to show that for any $1 \ne x \in K$ the centralizer $C_G (x)$ is contained in $K$.  Indeed, assume that there exists an element $1 \ne x \in K$ such that $C_G (x) \not\leq K$.  Then there exist $k \in K$ and $h \in H \setminus \{1\}$ such that $[kh,x]=1$.  As a consequence $kh \in {\rm Nil}_G(x)$.  Since $K \leq {\rm Nil}_G(x)$ and ${\rm Nil}_G(x)$ is a subgroup, we have $K \langle h \rangle \leq {\rm Nil}_G(x)$.  Now, the fact that $H$ and $K$ have coprime orders implies that $h$ centralizes $K$, which is the Fitting subgroup of $G$. Therefore $h \in C_G (K) \leq K$ leads to a contradiction. 
\end{proof}

We close this section with the following question.  We have seen that $A$-groups and Frobenius groups whose Frobenius complements are solvable $\mathfrak n$-groups as are their direct products.  Recall that by modding out by the hypercenter, we may consider centerless solvable groups.  Thus, we ask whether any centerless solvable group $G$ that is an $\mathfrak n$-group will be a subgroup of a direct product $A \times F$ where $A$ is an $A$-group and $F$ is the direct product of Frobenius groups whose Frobenius complements are $\mathfrak n$-groups.

\section{Connectivity of $\Gamma (G)$}

\noindent The goal of this section is to characterize the groups where $\Gamma (G)$ is disconnected. We start pointing out the connection between adjacency in $\Gamma (G)$ and $\Gamma (G/\hyp(G))$, for a group $G$.

\begin{pro}\label{prop:adj}
Let $G$ be a group and $x,y \in G \setminus \hyp(G)$ such that $x \hyp(G) \neq y \hyp(G)$. Then $x$ and $y$ are adjacent in $\Gamma(G)$ if and only if $x\hyp(G)$ and $y\hyp(G)$ are adjacent in $\Gamma(G/\hyp(G))$.
\end{pro}

\begin{proof}
If $x$ and $y$ generate a nilpotent subgroup, then also the subgroup generated by $x\hyp(G)$ and $y\hyp(G)$ is clearly nilpotent. 

Now assume that, modulo $\hyp(G)$, $x$ and $y$ generate a nilpotent subgroup $H$ of $G$. Thus there exists a positive integer $c$ such that $$[\underbrace{H, ..., H}_{c \mbox{ times}}] \leq \hyp(G).$$ Hence, for a suitable $t$ we have 
\begin{eqnarray*}
   [\underbrace{H, ..., H}_{c \mbox{ times}}, \underbrace{H, ..., H}_{t \mbox{ times}}] &\leq&  [\hyp(G), \underbrace{G, ..., G}_{t \mbox{ times}}]=1.
\end{eqnarray*}
Therefore, $H$ is nilpotent.
\end{proof}

\begin{thmB*}
For any group $G$ the number of connected components of $\Gamma(G)$ equals the number of connected components of $\Gamma(G/\hyp(G))$, and there is a correspondence between the connected components of $\Gamma(G)$ and $\Gamma(G/\hyp(G))$ that maps connected components of diameter $1$ to connected components of diameter $0$ or $1$ and preserves the diameter of connected components whose diameter is greater than $1$.
\end{thmB*}

\begin{proof}
We first observe that when $x$ and $y$ are adjacent in $\Gamma (G)$, by Proposition \ref{prop:adj} either $x\hyp (G) = y\hyp (G)$ or $x\hyp (G)$ and $y \hyp (G)$ are adjacent in $\Gamma (G/\hyp (G))$.  

Suppose $d \ge 2$.  We show that there exists a path of length $d$ between $x$ and $y$ in $\Gamma (G)$ if and only if there is a path of length $d$ between $x \hyp (G)$ and $y \hyp (G)$ in $\Gamma (G/\hyp (G))$.   To see this suppose $x, y \in G \setminus \hyp(G)$ such that the distance between $x$ and $y$ is $d$. Since $d \geq 2$, $x\hyp(G) \neq y\hyp(G)$ and there exist $x_1, \ldots, x_d+1 \in G \setminus \hyp(G)$ such that $x_1=x, x_d+1=y$ and $x_i$ is adjacent to $x_{i+1}$ for every $i \in \{1, \ldots, d\}$. Then $x_1\hyp(G), \ldots, x_{d+1}\hyp(G)$ is a path in $\Gamma(G/\hyp(G))$ of length at most $d$. Actually, for every $i \in \{1, \ldots d\}$, $x_i\hyp(G) \neq x_{i+1}\hyp(G)$. For, assume that $x_j\hyp(G) =x_{j+1}\hyp(G)$ for some $j \in \{1, \ldots, d\}$. Then, by Proposition \ref{prop:adj}, removing an element from $\{x_1, \ldots, x_{d+1}\}$ we find a path in $\Gamma(G)$ connecting $x$ and $y$ of length less than $d$, which is a contradiction.  This proves the observation.

It follows that $x$ and $y$ are in the same connected component of $\Gamma (G)$ if and only if $x\hyp (G)$ and $y \hyp (G)$ are in the same connected component of $\Gamma (G/\hyp (G))$.  This gives the bijection between connected components.  

Notice that a connected component of diamter $1$ in $\Gamma (G)$ can correspond to either a connected component that is a complete graph or one that has diamter $1$ in $\Gamma (G/\hyp (G))$.  For connected components of bigger diameter, we see that the diameter is preserved.
 \end{proof}

We obtain the following corollary for nilpotent graphs that are connected.

\begin{cor}
If $G$ is a group, then $\Gamma (G)$ is connected if and only if $\Gamma (G/\hyp (G))$ is connected.  
\end{cor}

For solvable groups, we can determine precisely when this occurs.  We denote by ${\rm Fit} (G)$ the Fitting subgroup of a group $G$.  Recall that a group $G$ is a {\it $2$-Frobenius group} if it has normal subgroups $K$ and $L$ such that $L$ is a Frobenius group with Frobenius kernel $K$ and $G/K$ is a Frobenius group with Frobenius kernel $L/K$ (see for instance \cite{CL-cycl}).

First, we show the following result.

\begin{pro}\label{prop:frob2frob}
If $G$ is a Frobenius group or a $2$-Frobenius group, then $\Gamma (G)$ is disconnected.
\end{pro}

\begin{proof}
First assume that $G$ is a Frobenius group. As the center of $G$ is trivial, any nontrivial element is a vertex in $\Gamma (G)$. Denote by $K$ the Frobenius kernel of $G$. Since $K$ is nilpotent, $K \setminus \{ 1 \}$ is contained in a connected component of $G$. Actually, $K\setminus \{ 1 \}$ is a connected component of $\Gamma (G)$.  Indeed, assume by way of contradiction that there exist an element $x \in K \setminus \{ 1 \}$ and an element $y \in G \setminus K$ such that $H = \langle x, y \rangle$ is nilpotent. Since $x$ and $y$ have coprime order, they commute. Thus $y \in C_G (x)$ which is a contradiction.

Now assume that $G$ is a $2$-Frobenius group, and let $F_1, F_2$ be subgroups of $G$ such that  $F_1 = {\rm Fit} (G)$ and $F_2/F_1 = {\rm Fit} (G/F_1)$. Then $F_1$ is the Frobenius kernel of $F_2$ and $F_2/F_1$ is the Frobenius kernel of $G/F_1$. It is sufficient to prove that an element in $F_2\setminus F_1$ cannot be adjacent to any nontrivial element outside $F_2 \setminus F_1$ in $\Gamma (G)$.  Consider an element $x \in F_2 \setminus F_1$. Since $F_1$ is the Frobenius Kernel of $F_2$, then the previous argument implies that for every element $y \in F_1 \setminus \{ 1 \}$ the elements $x$ and $y$ are not adjacent in $\Gamma (G)$.  Assume by contradiction that there exists an element $z \in G\setminus F_2$ such that $x$ and $z$ are adjacent in $\Gamma (G)$. Then $xF_1$ and $zF_1$ are also adjacent in $\Gamma(G/F_1)$. Thus there exists an element $t \in G \setminus F_1$ such that $tF_1$ centralizes both $xF_1$ and $zF_1$. This is a contradiction since $G/F_1$ is a Frobenius group.
\end{proof}

Now, we show that for a solvable group with trivial center the converse of the previous result is also true.  Recall that Parker shows in \cite{parker} that the commuting graph of a solvable group $G$ with a trivial center is disconnected if and only if $G$ is a Frobenius group or a $2$-Frobenius group.

\begin{pro}\label{prop:commuting}
Let $G$ be a group with trivial center.
\begin{enumerate}[label=$(\arabic*)$]
\item 
$\Gamma (G)$ is connected if and only if $\Gamma_{\rm comm} (G)$ is connected.  Moreover, if $\Gamma(G)$ is connected of diameter $k$, then the $diam(\Gamma_{\rm comm}(G)) \leq 2k$.
\item 
When $G$ is solvable, $\Gamma(G)$ is disconnected if and only if $G$ is a Frobenius group or a $2$-Frobenius group.
\end{enumerate}
\end{pro}

\begin{proof}
First observe that $\Gamma (G)$ and $\Gamma_{\rm comm} (G)$ have the same vertex set since $Z(G)$ is trivial.  To prove (1), notice that if $\Gamma_{\rm comm} (G)$ is connected, then clearly $\Gamma (G)$ is connected also. On the other hand, assume that $\Gamma (G)$ is connected. 
Suppose $x, y \in G \setminus \{ 1 \}$, then  there exist $k \in \N$ and $x_1, \ldots , x_{k-1} \in G \setminus \{ 1 \}$ such that $\langle x_i , x_{i+1} \rangle$ is nilpotent, with $x_0=x$, $x_{k}=y$, and $0 \leq i \leq k-1$. Therefore, for every integer $i \in \{0, \ldots, k-1\}$ there is a non-trivial central element $z_i$ in $\langle x_i , x_{i+1} \rangle$. % such that $z_i \in C_G(x_i)\cap C_G(x_{i+1})$.
Observe that $x = x_0$, $z_0$, $x_1$, $z_1$, $x_1$, $\dots$, $z_{k-1}$, $x_k$ is a path of length at most $2k$ in $\Gamma_{\rm comm} (G)$.  Thus $x$ and $y$ are connected by a path in $\Gamma_{\rm comm} (G)$ of length at most $2k$.

In order to prove (2), by Theorem 1.1 (i) of \cite{parker}, $\Gamma_{\rm comm}(G)$ is disconnected if and only if $G$ is a Frobenius group or a $2$-Frobenius group. Now, applying (1), we obtain the desired result.
\end{proof}

It is worth mentioning that if we suppose $Z (G) \neq \hyp (G)$, then $\Gamma (G)$ and $\Gamma_{\rm comm} (G)$ are always different because they have a different sets of vertices.  We now extend Proposition~\ref{prop:commuting} (2) to the case where $Z (G)$ is nontrivial.

\begin{thmC*}\label{thm:disconnected}
Let $G$ be a non-nilpotent solvable group. Then $\Gamma (G)$ is disconnected if and only if $G/\hyp (G)$ is a Frobenius group or a $2$-Frobenius group.
\end{thmC*}

\begin{proof}
First assume that $\Gamma (G)$ is disconnected, that is there exist elements $x, y \in G \setminus \hyp(G)$ which are not connected by any path in $\Gamma(G)$.  By Proposition \ref{prop:commuting} (2), since $G/\hyp(G)$ has trivial center it is sufficient to show that $\Gamma(G/\hyp(G))$ is disconnected to obtain the conclusion. 

By way of contradiction, assume that $\Gamma (G/\hyp (G))$ is connected. Notice that $x\hyp (G) \neq y \hyp(G)$ otherwise $\langle x, y \rangle \hyp(G)$ would be cyclic, which yields $\langle x,y \rangle$ nilpotent. This is impossible by the choice of $x$ and $y$.  Therefore there exists a path in  $\Gamma(G/\hyp(G))$ connecting $x\hyp(G)$ and $y\hyp(G)$, that is, there exist $n \in \N$ and $x_1, \ldots , x_n \in G \setminus \hyp (G)$ such that $\langle x_i \hyp (G), x_{i+1}\hyp (G) \rangle$ is nilpotent, with $x_0=x$, $x_{n+1}=y$ and $0 \leq i \leq n$. By Proposition \ref{prop:adj} $ x_i \hyp(G)$ and $x_{i+1}\hyp(G)$ are adjacent in $\Gamma (G/\hyp (G))$ if and only if $x_i$ and $x_{i+1}$ are adjacent in $\Gamma(G)$, therefore $x_0, \ldots , x_{n+1}$ realizes a path in $\Gamma (G)$ connecting $x$ and $y$, which is a contradiction.
    
Now, assume that $G/\hyp(G)$ is either a Frobenius group or a $2$-Frobenius group.  Then $\Gamma(G/\hyp(G))$ is disconnected by Proposition \ref{prop:commuting}, and there exist $x\hyp(G), y\hyp(G) \in G/\hyp(G)$ which are not connected. By contradiction assume that $\Gamma (G)$ is connected. Therefore there exist $n \in \N$ and $x_1, \ldots , x_n \in G \setminus \hyp(G)$ such that $\langle x_i , x_{i+1} \rangle$ is nilpotent, with $x_0=x$, $x_{n+1}=y$ and $0 \leq i \leq n$.  By Proposition \ref{prop:adj} $ x_i \hyp(G)$ and $x_{i+1}\hyp(G)$ are adjacent in $\Gamma(G/\hyp(G))$ if and only if $x_i$ and $x_{i+1}$ are adjacent in $\Gamma(G)$. This implies that $x_0 \hyp(G), \ldots, x_{n+1}\hyp(G)$ realizes a path in $\Gamma(G/\hyp(G))$ connecting $x\hyp(G)$ and $y\hyp(G)$, which is a contradiction.
\end{proof}

\section{Diameter of $\Gamma(G)$}

\noindent In this section we will provide some upper bounds for the diameter ${\rm diam} (\Gamma (G))$ of the graph $\Gamma (G)$. More precisely, we will consider non-nilpotent groups $G$ such that $\Gamma (G)$ is connected, and we will further assume that $G$ has trivial center due to the following result.

As a consequence of Theorem B and Theorem C we may assume that $G$ is neither Frobenius nor $2$-Frobenius. Going further, we can apply a result of Morgan and Parker to obtain items (1) and (2) of  Theorem E.  One should compare this result with Proposition 7.6 of \cite{BLN}.

\begin{cor}\label{cor:bound}
Let $G$ be a non-nilpotent group. Then the connected components of $\Gamma (G)$ have diameter is at most $10$. Moreover, if $G$ is solvable and $\Gamma (G)$ is connected, then ${\rm diam} (\Gamma (G)) \leq 8$.
\end{cor}

\begin{proof}
By Theorem B, we have ${\rm diam} (\Gamma (G)) = {\rm diam} (\Gamma (G/\hyp(G)))$. Since $G/\hyp(G)$ has trivial center, Theorem 1.1 of \cite{MP} implies that the commuting graph of $G/\hyp(G)$ has diameter at most $10$. Observing that the commuting graph of $G/\hyp(G)$ is a subgraph of $\Gamma (G/\hyp (G))$, then ${\rm diam} (\Gamma (G)) \leq 10$ as ${\rm diam} (\Gamma (G)) = {\rm diam} (\Gamma (G/\hyp (G)))$.
    
Now assume that $G$ is solvable, using the same argument as before and applying Theorem 1.1 (ii) of \cite{parker}, we obtain the desired result.
\end{proof}

In \cite{parker}, Parker provided an example of a solvable group $H$ such that $\Gamma_{\rm comm} (H)$ has diameter $8$. However, $\Gamma (H)$ has diameter at most $7$, because any two vertices of $\Gamma (H)$ belonging in ${\rm Fit} (H)$ are adjacent.  This leaves open the question of whether there exists a solvable group $G$ such that ${\rm diam}(\Gamma(G)) = 8$.  Finally, we find the bound on the diameter of the connected components of $\Gamma (G)$ when $G$ is solvable and $\Gamma (G)$ is disconnected, which proves item (3) of Theorem E.

\begin{pro}
Let $G$ be a solvable group and suppose that $\Gamma (G)$ is disconnected.  Then the diameter of one connected component is at most $5$ and the other connected components have diameters at most $2$.
\end{pro}

\begin{proof}
Without loss of generality, we may assume that $\hyp (G) = 1$, thus $G$ is either a Frobenius group or a $2$-Frobenius group.  Let $F_1$ be the Fitting subgroup of $G$ ahd $F_2/F_1$, the Fitting subgroup of $G/F_1$.  We know that $F_2$ is a Frobenius group in either case, and if $H$ is a Frobenius complement for $F_2$, then $H \setminus \{ 1 \}$ is a connected component of $\Gamma (G)$.  We know $Z(H) > 1$ and we know every nonidentity element in $Z(H)$ will be adjacent to every element in $H \setminus \{ 1 \}$, so the connected component corresponding to $H \setminus \{ 1 \}$ will have diameter at most $2$.  Since every element in $F_2 \setminus F_1$ is conjugate to $H \setminus \{ 1 \}$, this accounts for all of the elements in $F_2 \setminus F_1$.  

If $G = F_2$, then the remaining connected component is $F_1 \setminus \{ 1 \}$ is a complete graph.  Thus, we may assume that we are in the case where $G$ is a $2$-Frobenius group, and the remaining connected component is $(G \setminus F_2) \cup (F_1 \setminus \{ 1 \})$.  We prove that every element in $G \setminus F_2$ is a distance at most $2$ from a point in $F_1 \setminus \{ 1 \}$, and this will prove the conclusion.  By the Frattini argument, we have that $G = F_1 N_G (H)$, and since every element in $F_2 \setminus F_1$ is conjugate to $H \setminus \{ 1 \}$, it follows that every element in $G \setminus F_2$ is conjugate to an element in $N_G (H) \setminus H$.  Thus, if $x \in G \setminus F_2$, then without loss of generality, we may assume that $x \in N_G (H) \setminus H$.  We know that $N_G (H)$ is a Frobenius group with Frobenius complement $C$ that is cyclic.  Thus, without loss of generality, we may assume that $x \in C$.  Now, we know that $C$ does not act Frobeniusly on $F_1$, so there exists $y \in C \setminus \{ 1 \}$ so that $C_{F_1} (y) > 1$.  This gives our path of distance at most $2$ from $x$ to a point in $F$.
\end{proof}

\noindent Now, we show that we can obtain a tighter bound on the diameter of $\Gamma (G)$ when we impose extra conditions on $G$.  Recall that the Fitting subgroup of a solvable group will be nontrivial, and so, the elements in ${\rm Fit} (G) \setminus \hyp (G)$ will lie in a single connected component of $\Gamma (G)$.  We now study this connected component.  We also show that if $Z(G)$ is trivial, all of the primes dividing $|G|$ divide $|{\rm Fit} (G)|$, and $\Gamma (G)$ is connected, then $\Gamma (G)$ has diameter at most $5$.

\begin{pro}\label{prop:coprime}
Let $G$ be a solvable group with trivial center. Then 
\begin{enumerate}[label=$(\arabic*)$]
\item ${\rm Fit} (G) \setminus \{1\}$ is contained in a connected component $\mathcal X$ of $\Gamma (G)$;
\item If $x$ is a non-trivial element of $G$ such that $(o(x), |{\rm Fit} (G)|) \neq 1$, then $x \in \mathcal X$ and there exists an element $w \in {\rm Fit} (G) \setminus \{1\}$ such that $d(x,w) \leq 2$.
\item If $\pi(G)=\pi({\rm Fit} (G))$ and $\Gamma (G)$ is connected, then ${\rm diam} (\Gamma (G)) \leq 5$. 
\end{enumerate}  
\end{pro}

\begin{proof}
As ${\rm Fit} (G)$ is a nontrivial, nilpotent subgroup of $G$, it is clear that ${\rm Fit} (G) \setminus \{ 1 \}$ is contained in a connected component $\mathcal X$ of $\Gamma (G)$, and (1) is clear.  Now, take an element $x \in G \setminus \{1\}$ and a prime $p$ such that $p$ divides $(o(x), |{\rm Fit} (G)|)$. Of course $x$ is adjacent to its $p$-component, $x_p$.  Let $P$ be a Sylow $p$-subgroup of $G$ containing $x_p$ and observe that $P \setminus \{1\}$ is a clique of $\Gamma(G)$.  Moreover, as $p$ divides $|{\rm Fit} (G)|$, there exists a nontrivial element $w \in P \cap {\rm Fit} (G)$. Thus, $d(x,w) \leq 2$ since both $x$ and $w$ are adjacent to $x_p$.   Consider elements $x,y \in G \setminus \{ 1 \}$. If $\pi (G) = \pi ({\rm Fit} (G))$, then applying conclusion (2), we can find $w_1, w_2 \in {\rm Fit} (G) \setminus \{ 1 \}$ so that $d(x,w_1) \le 2$ and $d (y,w_2) \le 2$.  Since $w_1$ and $w_2$ will be adjacent in $\Gamma (G)$, the result follows.
\end{proof}

We next show that if $\Gamma (G)$ is connected and $|G:{\rm Fit} (G)|$ is a prime, then $\Gamma (G)$ has diameter at most $3$.

\begin{pro}\label{prop:primeindex}
 Let $G$ be a non-nilpotent group and $F = {\rm Fit} (G)$ such that $|G:F| = p$ with $p$ prime. If $\Gamma (G)$ is connected, then ${\rm diam}(\Gamma (G)) \leq 3$.
\end{pro}

\begin{proof}
Assume first that the center of $G$ is trivial.  Since any two distinct elements of $F$ are adjacent, it is sufficient to prove that for every element $x \in G \setminus \{ 1 \}$ there exists $w \in F \setminus \{ 1 \}$ such that $d (x,w) = 1$.  Assume $x \not \in F$. If $x^p \neq 1$, then $x^p \in F \setminus \{ 1 \}$ and $d (x, x^p) = 1$. Thus, suppose that $x^p = 1$.  If $p$ does not divide $|F|$, then $G = F\langle x \rangle$, and it is not a Frobenius group since $\Gamma (G)$ is connected. Therefore there exists an element $w \in F \setminus \{ 1 \}$ such that $[x,w] = 1$.  If $p$ divides $|F|$, then there exists $P$ a Sylow $p$-subgroup of $G$ such that $x \in P$ and $P \cap F \neq \{ 1\}$.  It follows that there exists $w \in F \setminus \{ 1 \}$ such that $\langle x, w \rangle $ is nilpotent. 

Now assume that the center of $G$ is not trivial. Then by Theorem B, $\Gamma (G/\hyp (G))$ is connected and ${\rm diam} (\Gamma (G)) = {\rm diam} \Gamma (G/\hyp (G))$.  Notice that $F/\hyp(G))$ has prime index in  $G/\hyp(G)$. The statement is now clear.
\end{proof}

This bound is the best possible, in fact the group $G = $ SmallGroup (54,5) has Fitting Subgroup of order $27$ and $\Gamma (G)$ is connected of diameter $3$.

Recall that a group $G$ is said to be an $AC$-group if every non-trivial element of $G$ has abelian centralizer.  Solvable $AC$-groups have been classified by Schmidt in \cite{schmidt}.

\begin{pro}\label{prop:ac-group}
If $G$ is a solvable, non-nilpotent $AC$-group, then $\Gamma (G)$ is disconnected.  
\end{pro}

\begin{proof}
By \cite[Satz 5.12]{schmidt}, either $G$ has an abelian normal subgroup of prime index, $G/Z(G)$ is Frobenius or $2$-Frobenius, or $G$ is nilpotent. Assume that $G$ has an abelian normal subgroup $A$ of prime index. Clearly, $A$ is the Fitting subgroup of $G$ and $C_G (x) = A$ for every element $x \in A \setminus \{ 1 \}$. Indeed, $C_G(x) < G$ and $|G:A|$ is prime. Hence $G$ is a Frobenius group and Proposition \ref{prop:frob2frob} implies that $\Gamma (G)$ is disconnected.
     
If $G/Z(G)$ is Frobenius or $2$-Frobenius, then $Z (G) = hyp (G)$ and Theorem C implies that $\Gamma (G)$ is disconnected. Since $G$ is not nilpotent, we are done. 
\end{proof}

We now show that $A$-groups are the only groups where $\Gamma (G)$ and $\Gamma_{\rm comm} (G)$ coincide, which is completely coherent with Theorem \ref{thm:agroup}.

\begin{thmD*}\label{prop:agroup}
Let $G$ be a group. Then $\Gamma(G)$ coincides with $\Gamma_{\rm comm}(G)$ if and only if $G$ is an $A$-group.
\end{thmD*}

\begin{proof}
Suppose that $\Gamma (G)$ coincides with $\Gamma_{\rm comm} (G)$. Clearly the center of $G$ coincides with the hypercenter of $G$. Let $P$ be a Sylow $p$-subgroup of $G$. If $P$ is contained in the center of $G$, then it is abelian. Thus, assume that $P$ is not contained in the center of $G$ and let $x, y \in P \setminus Z(G)$. Since $\langle x,y\rangle$ in nilpotent, it follows that $x$ and $y$ are adjacent in $\Gamma(G)$. Thus, by the hypothesis, they also commute and therefore $P$ is abelian.
    
Conversely, assume that $G$ is an $A$-group.  As a consequence, any nilpotent subgroup of $G$ is abelian.  We now show that the hypercenter of $G$ coincides with the center $Z(G)$ of $G$ because if $x \in \hyp(G)$ then for every $y \in G$ we have $\langle x, y \rangle$ is nilpotent, thus $[x,y]=1$ and $x \in Z(G)$. Hence $\Gamma(G)$ and the commuting graph of $G$ have the same set of vertices. Moreover any two vertices $x$ and $y$ are adjacent in $\Gamma(G)$ if and only if they are adjacent in $\Gamma_{\rm comm} (G)$. This concludes the proof.
\end{proof}

From Theorem D and \cite[Theorem 1.1]{CL-agrp} it follows that for an $A$-group $G$ the diameter of $\Gamma(G)$ is at most $6$.

Next, we show that if $G$ is non-nilpotent and has the form $N \rtimes A$ with $Z(G) = \hyp (G)$ where $N$ is cyclic and $A$ is abelian, then $\Gamma (G)$ has diameter at most $4$.

\begin{pro}\label{prop:cyclic-by-ab}
Let $G = N \rtimes A$ a non-nilpotent group with $N$ cyclic and $A$ abelian. If $\Gamma (G)$ is connected and $Z (G) = \hyp (G)$, then $\Gamma (G)$ has diameter at most $4$. 
\end{pro}

\begin{proof}
From $\Gamma(G)$ connected and Proposition \ref{prop:ac-group}, it follows that $G$ is not an $AC$-group.  Therefore, there exists a noncentral element $x \in G$ such that $C_G(x)$ is non-abelian. By Theorem 1.2 (a) of \cite{velten}, $\Gamma_{\rm comm}(G)$ is connected and by Theorem 1.2 (b) of \cite{velten}, it has diameter at most $4$. Therefore, also $\Gamma(G)$ has diameter at most $4$. 
\end{proof}

We point out that this bound is sharp. Indeed the group $G = \langle x \rangle \rtimes \langle y \rangle$ with $x^{15} = y^4 = 1$ and $x^y = x^8$ satisfies the hypothesis of the previous result and $\Gamma (G)$ is connected of diameter $4$.

We now show that if ${\rm Fit} (G)$ is cyclic, then $\Gamma (G)$ has diameter at most $5$.

\begin{pro}\label{prop:cyclicfit}
Let $G$ be a solvable group with trivial center such that $\Gamma(G)$ is connected.   If the Fitting subgroup of $G$ is cyclic, then the diameter of $\Gamma(G)$ is at most $5$.
\end{pro}
    
\begin{proof}
Fix an element $x \in G\setminus \{ 1 \}$ and write $F = {\rm Fit} (G)$.  It is sufficient to prove that there exists an element $f \in F \setminus \{ 1 \}$ such that $x$ is connected to $f$ in at most $2$ steps.

If $(o(x), |F|) \neq 1$, then by Proposition \ref{prop:coprime} (2) the result follows.  Suppose $(o(x), |F|)=1$, and let $p$ be a prime that divides $o(x)$.  Denote with $x_p$, the $p$-component of $x$. Obviously $x$ and $x_p$ are adjacent in $\Gamma(G)$.  Let $P$ be a Sylow $p$-subgroup of $G$ such that $x_p \in P$.  Notice that $G/F$ is isomorphic to a subgroup of ${\rm Aut} (F)$ which is abelian because $F$ is cyclic; therefore $G/F$ is abelian.

Thus, $FP$ is normal in $G$ and by the Frattini argument, it follows that $G = N_G(P)F$.  If $N_G (P) \cap F = \{ 1 \}$, then $G$ has a cyclic-by-abelian factorization, and by Proposition \ref{prop:cyclic-by-ab} the diameter of $\Gamma (G)$ is at most $4$.  Suppose $N_G(P) \cap F \neq \{ 1 \}$.  Notice that $N_G(P) \cap F$ is normal in $N_G (P)$, and thus, $N_G (P) \cap F \subseteq {\rm Fit} (N_G(P))$.  Moreover, $P \subseteq {\rm Fit} (N_G(P))$. Therefore, there exists $f \in F \setminus \{ 1 \}$ such that $\langle x_p, f \rangle$ is nilpotent. The conclusion is now clear.
\end{proof}

We now show that if $G$ is a $\{ p,q \}$-group, then $\Gamma (G)$ has diameter at most $6$.

\begin{thm}\label{thm:pq}
Let $G$ be a non-nilpotent $\{p, q\}$-group with trivial center. If $\Gamma (G)$ is connected, then ${\rm diam} (\Gamma (G)) \leq 6$.
\end{thm}

\begin{proof}
It is sufficient to show that for any element $x \in G \setminus \{ 1 \}$ the distance between $x$ and any element of $F = {\rm Fit}(G)$ is at most $3$.   If $\pi (F) = \pi (G)$, then by Proposition \ref{prop:coprime} (3) the result follows. Therefore, without loss of generality, we may assume that $q$ does not divide $|F|$, and so,  $F$ is a $p$-group.

Consider an element $x \in G \setminus \{ 1 \}$ and $P \in {\rm Syl}_p (G)$.  Assume first that $P$ is normal in $G$; that is, $P = O_p (G) = F$.  If $p$ divides $o(x)$, then $x$ is adjacent in $\Gamma (G)$ to its $p$-component, say $x_p$, which lies in $F$.  If $p$ does not divide $o(x)$, then $x$ is a $q$-element and there exists $Q \in {\rm Syl}_q(G)$ such that $x \in Q$.  Moreover, notice that $G = FQ$ which is not a Frobenius group due to the fact that $\Gamma (G)$ is connected.  Thus, there exist elements $w \in Q \setminus \{ 1 \}$ and $f \in F\setminus \{ 1 \}$ such that $[w,f] = 1$. Therefore, $x$ is connected to an element of $F$ in at most $2$ steps.

Assume now that $O_p(G) = F < P$. If $p$ divides $o(x)$, then $x$ is adjacent to its $p$-component, say $x_p$, which is adjacent to a nonidentity element of $F$.  If $p$ does not divide $o(x)$, then $x$ is a $q$-element. If $p$ divides $|C_G (x)|$ then there exists an element $w \in C_G(x)$ such that $w$ is a $p$-element, and so, $x$ is connected to an element of $F$ in at most $2$ steps. Therefore, we may assume that $p$ does not divide $|C_G (x)|$, and thus, that there exists $Q \in {\rm Syl}_q (G)$ such that $C_G (x) \leq Q$. If there exist elements $f \in F \setminus \{ 1\}$ and $y \in Q \setminus \{ 1 \}$ such that $[f,y]=1$, then $x$ is connected to $f$ in at most $2$ steps. Thus, we may assume for every $f \in F \setminus \{ 1 \}$ and for every $y \in Q \setminus \{ 1 \}$ that we have $[f,y] \neq 1$. Now, we wish to show that $FQ$ is a Frobenius group.  Let $v \in Q \setminus \{ 1 \}$ and $z \in C_{FQ}( v)$. There exist $f \in F$ and $g \in Q$ such that $z=fg$. We have
$$
1 = [z,v] = [fg,v] = [f,v]^g[g,v].
$$
Notice that $[g,v] \in Q$, $[f,v] = f^{-1} f^v \in F$ and so also $[f,v]^g \in F$. Therefore, $[f,v] = 1$ which yields to $f = 1$ and $z \in Q$. Thus, $C_{FQ}(v) \leq Q$ and $FQ$ is a Frobenius group.
Denote $K/F = O_q (G/F)$. Obviously $K/F$ is a subgroup of $FQ/F$. By Theorem 9.3.1 of \cite{Robinson}, we have $C_G(K/F)\leq K$, and thus, $C_{G/F} (K/F) \leq K/F$.
 
If $q$ is odd then the fact that $FQ$ is a Frobenius group yields that $Q$ is cyclic and $FQ/F \cong Q$ is contained in $C_{G/F} (K/F)$, therefore $C_{G/F} (K/F) \leq K/F \leq FQ/F \leq C_{G/F} (K/F)$, which implies that $FQ=K$ is normal in $G$.  Obviously $Q$ is a $q$-Sylow subgroup of $FQ$; so therefore by the Frattini argument, it follows that $G = N_G(Q) FQ = N_G(Q)F$.  Since $\Gamma(G)$ is connected, $G$ is not a $2$-Frobenius group, and so, $G/F$ is not a Frobenius group.  Thus $N_G(Q)$ is not a Frobenius group, so there exist elements $y \in C_{N_G(Q)}(x)\setminus \{1\}$ and $g \in C_{N_G(Q)}(y)\setminus Q$, with $o(g)=p$  such that for every $w \in F$ we have $x \sim y \sim g \sim w$.
 
Therefore $x$ is connected to any nonidentity element of $F$ in at most $3$ steps.  If $q=2$, then we may assume that $Q$ is not cyclic, otherwise we can argue as before.  Thus, the fact that $FQ$ is a Frobenius group yields that $Q$ is generalized quaternion.  Again 
$$C_{G/F}(K/F) \leq K/F;$$ 
moreover $\displaystyle \frac{G/F}{C_{G/F}(K/F)}$ is isomorphic to a subgroup of ${\rm Aut} (K/F)$.  Suppose that $K/F$ is cyclic or generalized quaternion of order at least $16$, it follows that ${\rm Aut} (K/F)$ is a $2$-group and $P/F \leq C_{G/F}(K/F) \leq K/F$, therefore $P/F=1$ and so $P = F$ which is a contradiction.  Then $Q\cong Q_8$ and $K/F \cong Q$.  Since ${\rm Aut} (Q_8)$ is the symmetric group of degree $4$, it follows that $p=3$ and $G/QF \cong C_3$ and so $G/F \cong SL_2(3)$.  Therefore, there exist $u \in F$ and $v \in O_2(G)$ such that $x \sim v \sim u$. 
\end{proof}

\section*{Acknowledgements}
\noindent This work was partially supported by the National Group for Algebraic and Geometric Structures, and their Applications  (GNSAGA -- INdAM). Funded by the European Union - Next Generation EU, Missione 4 Componente 1 CUP B53D23009410006, PRIN 2022- 2022PSTWLB - Group Theory and Applications.  The third author wishes to thank the Department of Mathematics of the University of Salerno, where much of this work was carried out, for the excellent hospitality.  The fourth  author wishes to thank the Department of Mathematical Sciences of Kent State  University, where a portion of this work was completed, for the excellent hospitality.

\bibliographystyle{plain}

\end{document}